\documentclass[12pt]{amsart}
\usepackage[latin1]{inputenc}
\usepackage{amsbsy}
\usepackage{amscd} 
\usepackage{amsfonts}
\usepackage{amsgen} 
\usepackage{amsmath}
\usepackage{amsopn} 
\usepackage{amssymb}
\usepackage{amstext}
\usepackage{amsthm} 
\usepackage{amsxtra}
\usepackage[all]{xy}
\usepackage{booktabs}
\usepackage{rotating}
\usepackage{tikz}
\usepackage{scalerel}

\theoremstyle{plain} 
\newtheorem{thm}{Theorem}[section]

\newtheorem{prop}[thm]{Proposition}
\newtheorem{lem}[thm]{Lemma}
\newtheorem{cor}[thm]{Corollary}
\theoremstyle{definition}
\newtheorem{defn}[thm]{Definition}
\newtheorem{rem}[thm]{Remark}

\numberwithin{equation}{section}

\renewcommand{\theta}{\vartheta}
\renewcommand{\phi}{\varphi}
\renewcommand{\epsilon}{\varepsilon}
\renewcommand{\subset}{\subseteq}

\newcommand{\N}{\mathbb N}

\newcommand{\Aut}{G_{aut}}

\newcommand{\QBan}{G_{aut}^+}
\newcommand{\QBic}{G_{aut}^*}

\begin{document}

\title{The Petersen Graph has no Quantum Symmetry}
\author{Simon Schmidt}
\address{Saarland University, Fachbereich Mathematik, 
66041 Saarbr\"ucken, Germany}
\email{simon.schmidt@math.uni-sb.de}
\thanks{The author is supported by the DFG project \emph{Quantumautomorphismen von Graphen}. He wants to thank his supervisor Moritz Weber for proofreading this article and for many helpful discussions on quantum automorphism groups of finite graphs.}
\date{\today}

\begin{abstract}
In 2007, Banica and Bichon asked whether the well-known Petersen graph has quantum symmetry. In this article, we show that the Petersen graph has no quantum symmetry, i.e. the quantum automorphism group of the Petersen graph is its usual automorphism group, the symmetric group $S_5$.
\end{abstract}

\maketitle

\section*{Introduction}

The study of graph automorphisms is an important branch of graph theory. The automorphism group of a finite graph on $n$ vertices without multiple edges is given by 
\begin{align*}
\Aut(\Gamma):=\{\sigma\in S_n\;\mid\;\sigma\epsilon=\epsilon\sigma\}\subset S_n,
\end{align*}
where $\epsilon$ is the adjacency matrix of the graph and $S_n$ denotes the symmetric group. 

As a generalization of this concept, quantum automorphism groups of finite graphs were introduced in the framework of compact matrix quantum groups. In 2005, Banica \cite{QBan} defined the quantum automorphism group $\QBan(\Gamma)$ based on the $C^*$-algebra
\begin{align*}
C(\QBan(\Gamma)):= C(S_n^+) / \langle u\epsilon=\epsilon u\rangle,
\end{align*}
where $S_n^+$ denotes the quantum symmetric group defined by Wang \cite{WanSn}. There is also another definition given by Bichon in \cite{QBic}, which is a quantum subgroup of Banica's quantum automorphism group.

In \cite{BanBic}, Banica and Bichon computed $\QBan(\Gamma)$ for all vertex-transitive graphs $\Gamma$ of order up to eleven, except the Petersen graph. The Petersen graph (see figure \ref{figure}) is a strongly regular graph on ten vertices and it often appears as a counter-example to conjectures in graph theory. For more information concerning the Petersen graph, we refer to the book \cite{Petersen}.

We say that a graph has no quantum symmetry if the quantum automorphism group of Banica coincides with the usual automorphism group of the graph, as defined in \cite{BanBic}. An important task is now to get graphs that do not have quantum symmetry. There is some work done for circulant graphs on $p$ vertices, where $p$ is a prime, in \cite{Che}, but there are no results in a more general setting.  

Our main result is that the Petersen graph has no quantum symmetry \linebreak(Theorem \ref{nosymmetry}). In our proof, we are mainly using the fact that the Petersen graph is strongly regular. Therefore, the question is if we can use a similar approach for other strongly regular graphs and get results for graphs like the Clebsch graph and the Highman-Sims graph as asked in \cite{survey}.

\section{Finite graphs and the Petersen Graph}
In this section, we fix some notations and recall some well-known facts about the Petersen graph. 

\subsection{Finite graphs} 
A graph $\Gamma=(V,E)$ is \emph{finite}, if the set $V$ of \emph{vertices} and the set $E$ of \emph{edges} are finite. We denote by $r:E\to V$ the \emph{range map} and by $s:E\to V$ the \emph{source map}. A graph is \emph{undirected} if for every $e \in E$ there is a $f \in E$ with $s(f) = r(e)$ and $r(f) = s(e)$, it is \emph{directed} otherwise. For a finite graph $\Gamma=(V,E)$ with $V=\{1,\ldots,n\}$, its \emph{adjacency matrix} $\epsilon\in M_n(\N_0)$ is defined as $\epsilon_{ij}:=\#\{e\in E\;|\; s(e)=i, r(e)=j\}$, where $\N_0=\{0,1,2,\ldots\}$. 

\subsection{The Petersen Graph}
The Petersen graph is a finite, undirected graph on ten vertices and is defined by the drawing in figure \ref{figure}. We will denote the Petersen graph by $\mathrm{P}$ in this article. 
\begin{figure}[h]
\begin{center}
\begin{tikzpicture}
\draw (18:1cm) -- (162:1cm) -- (306:1cm) -- (90:1cm) -- (234:1cm) -- cycle;
\draw (18:2cm) -- (90:2cm) -- (162:2cm) -- (234:2cm) -- (306:2cm) -- cycle;
\foreach \x in {18,90,162,234,306}
{\draw (\x:1cm) -- (\x:2cm);
\draw[black,fill=black] (\x:2cm) circle (2pt);
\draw[black,fill=black] (\x:1cm) circle (2pt);}
\end{tikzpicture}
\end{center}
\caption{\label{figure}The Petersen Graph}
\end{figure}
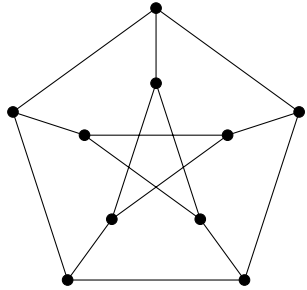

\subsection{Strongly regular graphs}
Now, let $\Gamma$ be an undirected graph. Let $v \in V$. The vertex $u \in V$ is called \emph{neighbor} of $v$, if $(v,u) \in E$. The \emph{degree} $\mathrm{deg}\, v$ of a vertex $v \in V$ denotes the number of edges in $\Gamma$ incident with $v$. We say that a graph $\Gamma$ is $k$-\emph{regular} for some $k \in \N_0$, if $\mathrm{deg}\, v = k$ for all $v \in V$. As we can see in figure \ref{figure}, the Petersen graph is $3$-regular. 

\begin{defn}
Let $\Gamma = (V,E)$ be a $k$-regular graph on $n$ vertices. We say that $\Gamma$ is \emph{strongly regular} if there exist $\lambda$, $\mu \in \N_0$ such that
\begin{itemize}
\item[$(i)$] adjacent vertices have $\lambda$ common neighbors,
\item[$(ii)$] non-adjacent vertices have $\mu$ common neighbors.
\end{itemize}
\end{defn}

\begin{prop}\cite{Petersen}
The Petersen Graph is a strongly regular graph with $\lambda = 0$, $\mu = 1$.
\end{prop}

\begin{cor}\label{ex1}
Hence, for the Petersen graph, one can write $(i), (ii)$ as follows
\begin{itemize}
\item[$(1)$] If $(i,k) \in E, (l,k) \in E$ and $i\neq l$, then $(i,l) \notin E$,
\item[$(2)$] If $(i,k) \notin E$ with $i \neq k$, then there exists exactly one $s \in V$ such that $(i,s)\in E$ and $(k,s) \in E$.
\end{itemize}
\end{cor}

\begin{rem}
There are only two other known strongly regular graphs with parameters $\lambda =0$, $\mu=1$, the cycle on $5$ vertices and the Hoffman-Singleton graph on $50$ vertices. 
\end{rem}

\subsection{Automorphism groups of finite graphs}
For a finite graph $\Gamma=(V,E)$ without multiple edges, a \emph{graph automorphism} is a bijective map $\sigma: V \to V$ such that $(\sigma(i), \sigma(j)) \in E$ if and only if $ (i,j) \in E$.
The set of all graph automorphisms of $\Gamma$ forms a group, the \emph{automorphism group} $\Aut(\Gamma)$. We can view $\Aut(\Gamma)$ as a subgroup of the symmetric group $S_n$, if $\Gamma$ has $n$ vertices
\begin{align*}
\Aut(\Gamma) = \{ \sigma \in S_n\;|\; \sigma \varepsilon = \varepsilon \sigma\}\subset S_n.
\end{align*} 
In case of the Petersen graph $\mathrm{P}$, it is well-known that $\Aut({\mathrm{P}}) = S_5$. 

\section{Quantum automorphism groups of finite graphs}

Now, we give the definitions of quantum automorphism groups of finite graphs. But first, we need the definition of compact matrix quantum groups which were defined by Woronowicz \cite{CMQG1,CMQG2} in 1987. 

\begin{defn}
A \emph{compact matrix quantum group} $G$ is a pair $(C(G),u)$, where $C(G)$ is a unital (not necessarily commutative) $C^*$-algebra which is generated by $u_{ij}$, $1 \leq i,j \leq n$, the entries of a matrix $u \in M_n(C(G))$. Moreover, the *-homomorphism $\Delta: C(G) \to C(G) \otimes C(G)$, $u_{ij} \mapsto \sum_{k=1}^n u_{ik} \otimes u_{kj}$ must exist, and $u$ and its transpose $u^{t}$ must be invertible. 
\end{defn}

In 2003, Bichon \cite{QBic} defined a quantum automorphism group as follows.

\begin{defn}
Let $\Gamma = (V, E)$ be a finite graph on $n$ vertices $V = \{1, ... , n \}$. The \emph{quantum automorphism group} $\QBic(\Gamma)$ is the compact matrix quantum group $(C(\QBic(\Gamma)), u)$, where $C(\QBic(\Gamma))$ is the universal $C^*$-algebra with generators $u_{ij}$, $1 \leq i,j \leq n$ and relations
\begin{align}\allowdisplaybreaks
&u_{ij} = u_{ij}^*, \quad u_{ij}u_{ik} = \delta_{jk}u_{ij} ,\quad  u_{ji}u_{ki} = \delta_{jk}u_{ji}, &&1 \leq i,j,k \leq n,\label{QA1}\\ 
&\sum_{l=1}^n u_{il} = 1 = \sum_{l=1}^n u_{li}, &&1 \leq i \leq n,\label{QA2}\\
&u_{ji}u_{lk} = u_{lk}u_{ji} = 0, && (i,k) \notin E, (j,l) \in E,\label{QA3}\\
&u_{ij}u_{kl} = u_{kl}u_{ij} = 0, && (i,k) \notin E, (j,l) \in E,\label{QA4}\\
&u_{ij}u_{kl} = u_{kl}u_{ij}, &&(i,k), (j,l) \in E.\label{QA5}
\end{align} 
\end{defn}

Two years later, Banica \cite{QBan} gave the following definition.

\begin{defn}
Let $\Gamma =(V, E)$ be a finite graph with $n$ vertices and adjacency matrix $\varepsilon \in M_n(\{0,1\})$.  The \emph{quantum automorphism group} $\QBan(\Gamma)$ is the compact matrix quantum group $(C(\QBan(\Gamma)),u)$, where $C(\QBan(\Gamma))$ is the universal $C^*$-algebra with generators $u_{ij}, 1 \leq i,j \leq n$ and Relations \eqref{QA1}, \eqref{QA2} together with
\begin{align}
u \varepsilon = \varepsilon u \label{QA7},
\end{align}
which is nothing but $\sum_ku_{ik}\epsilon_{kj}=\sum_k\epsilon_{ik}u_{kj}$.
\end{defn}

The next lemma will give a link between the two definitions. The proof of this lemma can be found in \cite[Lemma 3.1.1]{Ful} or \cite[Lemma 6.7]{SW16}.

\begin{lem}\label{link}
Let $\Gamma$ be a finite graph. It holds 
\[C(\QBan(\Gamma)) = C^*( u_{ij} \; | \; \textnormal{Relations  \eqref{QA1} -- \eqref{QA4}})\]
and therefore we have
\begin{align*}
\Aut(\Gamma) \subseteq \QBic(\Gamma) \subseteq \QBan(\Gamma)
\end{align*}
in the sense of compact matrix quantum subgroups.
\end{lem}

The next definition is due to Banica and Bichon \cite{BanBic}.

\begin{defn}
Let $\Gamma = (V,E)$ be a finite graph. We say that $\Gamma$ has \emph{no quantum symmetry} if $C(\QBan(\Gamma))$ is commutative, or equivalently
\begin{align*}
C(\QBan(\Gamma)) = C(\Aut(\Gamma)).
\end{align*}
\end{defn} 

For more on quantum automorphism groups of finite graphs, see \cite{SWe}.

\section{The Quantum Automorphism group of the Petersen Graph}

In this section, we show that the Petersen Graph $\mathrm{P}$ has no quantum symmetry. We use the following lemma.

\begin{lem}\label{com}
Let $(u_{ij})_{1 \leq i,j \leq n}$ be the generators of $C(\QBic(\Gamma))$ or $C(\QBan(\Gamma))$. If we have 
\begin{align*}
u_{ij}u_{kl} = u_{ij}u_{kl}u_{ij}
\end{align*}
 then $u_{ij}$ and $u_{kl}$ commute.
\end{lem}

\begin{proof}
It holds 
\begin{align*}
u_{ij}u_{kl}= u_{ij}u_{kl}u_{ij}= (u_{ij}u_{kl}u_{ij})^* = (u_{ij}u_{kl})^*= u_{kl}u_{ij}.
\end{align*}
\end{proof}

The next theorem shows that $\QBan(\mathrm{P})= \QBic(\mathrm{P})$. We state it slightly more general. 

\begin{thm}\label{BaneqBic}
Let $\Gamma$ be an undirected finite graph fulfilling the conditions $(1)$ and $(2)$ of Corollary \ref{ex1}. Then we have $\QBan(\Gamma) = \QBic(\Gamma)$. 
\end{thm}

\begin{proof}
We know from Lemma \ref{link} that
\begin{align*}
C(\QBan(\Gamma)) = C^*( u_{ij} \; | \; \textnormal{Relations  \eqref{QA1} -- \eqref{QA4}})
\end{align*}
and therefore we only have to show that Relation \eqref{QA5} is fulfilled in $C(\QBan(\Gamma))$. 
Let $(i,k) \in E$, $(j,l) \in E$. It holds
\begin{align*}
u_{ij}u_{kl} = u_{ij}u_{kl}\left(\sum_{s;(l,s) \in E} u_{is}\right)
 \end{align*} 
by Relations \eqref{QA2} and \eqref{QA3}. 

Take $s$ with $(l,s) \in E$, $s \neq j$. Since we also have $(j,l) \in E$, we get $(j,s) \notin E$ by $(1)$ of Corollary \ref{ex1}. Thus, by $(2)$ of Corollary \ref{ex1}, the only common neighbor of $j$ and $s$ is $l$. Hence, for all $a \neq l$, we have $(a,s) \notin E$ or $(a,j) \notin E$. Then Relation \eqref{QA3} implies $u_{ka}u_{is} =0$ or $u_{ij}u_{ka} = 0$ for all $a \neq l$. By also using Relations \eqref{QA1} and \eqref{QA2}, we get 
\begin{align*}
u_{ij}u_{kl}u_{is} = u_{ij}\left(\sum_{a=1}^n u_{ka}\right)u_{is}
=u_{ij}u_{is}
=0.
\end{align*} 
Therefore, we obtain
\begin{align*}
u_{ij}u_{kl} = u_{ij}u_{kl}\left(\sum_{s;(l,s) \in E} u_{is} \right) = u_{ij}u_{kl}u_{ij}
\end{align*}
and we conclude $u_{ij}u_{kl} = u_{kl}u_{ij}$ by Lemma \ref{com}. Thus, Relation \eqref{QA5} is fulfilled in $C(\QBan(\Gamma))$ and we get $\QBan(\Gamma) = \QBic(\Gamma)$.
\end{proof}

Now, we can prove our main result. 

\begin{thm}\label{nosymmetry}
The Petersen Graph $\mathrm{P}$ has no quantum symmetry, i.e. 
\begin{align*}
\QBan(\mathrm{P}) = \Aut(\mathrm{P}) = S_5.
\end{align*} 
\end{thm}

\begin{proof}
We show that $C(\QBan(\mathrm{P}))$ is commutative by using $(1),(2)$ of Corollary \ref{ex1} and the fact that $\mathrm{P}$ is $3$-regular. Let $(u_{ij})_{1 \leq i,j \leq 10}$ be the generators of $C(\QBan(\mathrm{P}))$. It suffices to show that 
\begin{align*}
u_{ij}u_{kl} = u_{kl}u_{ij}
\end{align*} 
for $(i,k) \notin E$, $(j,l) \notin E$, because we have
\begin{align*}
u_{ij}u_{kl} = 0 = u_{kl}u_{ij}
\end{align*}
for $(i,k) \notin E$, $(j,l) \in E$ or $(i,k) \in E$, $(j,l) \notin E$ by Relations \eqref{QA3}, \eqref{QA4} and we know that Relation \eqref{QA5} is fulfilled in $C(\QBan(\mathrm{P}))$ by Theorem \ref{BaneqBic}.

Let $(i,k) \notin E$, $(j,l) \notin E$, $i \neq k, j \neq l$ (it is obvious for $i=k$ or $j=l$). By $(2)$ of Corollary \ref{ex1} there exist exactly one $s \in E$ such that $(i,s) \in E, (k,s) \in E$ and exactly one $t \in E$ such that $(j,t) \in E, (l,t) \in E$.\\ 

\noindent\emph{Step 1: It holds $u_{ij}u_{kl} = u_{ij}u_{st}u_{kl}$.}\\
Since $t$ is the only common neighbor of $j$ and $l$, we have $(j,b) \notin E$ or $(b,l) \notin E$ for all $b \neq t$. Therefore, Relation \eqref{QA3} yields $u_{ij}u_{sb} =0$ or $u_{sb}u_{kl} =0$ for all $b \neq t$. By also using Relation \eqref{QA2}, we get 
\begin{align*}
u_{ij}u_{kl} = u_{ij}\left(\sum_{b=1}^n u_{sb}\right)u_{kl} = u_{ij}u_{st}u_{kl}.
\end{align*} 

\noindent\emph{Step 2: We have $u_{ij}u_{kl} = u_{ij}u_{st}u_{kl}(u_{ij} + u_{il} + u_{iq})$, where $q$ is the third neighbor of $t$.}\\
By using \emph{Step 1} and $u_{st}u_{kl} = u_{kl}u_{st}$ (by Relation \eqref{QA5}), we have $u_{ij}u_{kl} = u_{ij}u_{kl}u_{st}$.
Relations \eqref{QA2} and \eqref{QA3} imply
\begin{align*}
u_{ij}u_{kl} = u_{ij}u_{kl}u_{st}\left(\sum_{p;(t,p) \in E} u_{ip}\right).
\end{align*}
Since the Petersen graph is $3$-regular and since we know that $j$ and $l$ are neighbors of $t$, we get
\begin{align*}
u_{ij}u_{kl}&= u_{ij}u_{kl}u_{st}(u_{ij} + u_{il} + u_{iq}) = u_{ij}u_{st}u_{kl}(u_{ij} + u_{il} + u_{iq}),
\end{align*}
where we denote by $q$ the third neighbor of $t$.\\ 

\noindent\emph{Step 3: It holds $u_{ij}u_{st}u_{kl}u_{il} = 0$ and $u_{ij}u_{st}u_{kl}u_{iq} =0$.}\\
By Relation \eqref{QA1}, we have $u_{ij}u_{st}u_{kl}u_{il}  = 0$, because $k \neq i$. 
 For the second equation, observe that
\begin{align*}
u_{ij}u_{st}u_{kq}u_{iq} = 0 \quad \text{ and } \quad u_{ij}u_{st}u_{kj}u_{iq} = u_{ij}u_{kj}u_{st}u_{iq} = 0 
\end{align*}
by Relations \eqref{QA1} and \eqref{QA5}. We therefore get
\begin{align*}
u_{ij}u_{st}u_{kl}u_{iq} &= u_{ij}u_{st}(u_{kl}+u_{kj} + u_{kq})u_{iq}\\ 
&= u_{ij}u_{st}\left(\sum_{a;(t,a) \in E} u_{ka}\right)u_{iq}\\
&= u_{ij}u_{st}\left(\sum_{a=1}^n u_{ka}\right)u_{iq}\\
&= u_{ij}u_{st}u_{iq},
\end{align*}
where we also used Relations \eqref{QA2}, \eqref{QA3}. 
By Relation \eqref{QA1} and using $u_{ij}u_{st} = u_{st}u_{ij}$, we obtain
\begin{align*}
u_{ij}u_{st}u_{kl}u_{iq} = u_{ij}u_{st}u_{iq} = u_{st}u_{ij}u_{iq} = 0,
\end{align*}
since $j \neq q$.\\

\noindent\emph{Step 4: We  have $u_{ij}u_{kl} = u_{kl}u_{ij}$.}\\
From the previous steps, we conclude
\begin{align*}
u_{ij}u_{kl} = u_{ij}u_{kl}u_{ij}.
\end{align*}
Then Lemma \ref{com} yields $u_{ij}u_{kl} = u_{kl} u_{ij}$ and this completes the proof.
\end{proof}

\bibliographystyle{plain}
\bibliography{QAutPetersen}

\end{document}